\documentclass[12pt]{amsart}
\usepackage[margin=1in]{geometry}  
\usepackage[pdftex, dvipsnames]{xcolor}
\usepackage{graphicx}              

\usepackage{upgreek}
\usepackage{multirow}
\usepackage{amssymb}
\usepackage{amsmath}               
\usepackage{amsfonts}              
\usepackage{amsthm}                

\usepackage{tikz}
\usetikzlibrary{cd}
\usetikzlibrary{decorations.pathreplacing,intersections}
\usepackage[shortlabels]{enumitem}
\usepackage{overpic}
\usepackage{caption}
\usepackage{subcaption}
\usepackage[pdftex,%
    final,%
    colorlinks=true,%
    linkcolor=violet,%
    citecolor=violet,%
    filecolor=violet,%
    menucolor=violet,%
    urlcolor=violet,%
    bookmarks=true,%
    bookmarksdepth=3,%
    bookmarksnumbered=true,%
    bookmarksopen=true,%
    bookmarksopenlevel=2,%
]{hyperref}

\definecolor{specialteal}{RGB}{0, 128, 128}
\definecolor{specialorange}{RGB}{255, 153, 85}
\usepackage[textsize=tiny]{todonotes}
\usepackage{thm-restate}
\newcommand{\Z}{\mathbb Z}
\newcommand{\R}{\mathbb R}
\newcommand{\Q}{\mathbb Q}

\newtheorem{theorem}{Theorem}[section]
\newtheorem{lemma}[theorem]{Lemma}
\newtheorem{obstruction}[theorem]{Obstruction}

\newtheorem{proposition}[theorem]{Proposition}
\newtheorem{corollary}[theorem]{Corollary}
\newtheorem{conjecture}[theorem]{Conjecture}

\theoremstyle{remark}
\newtheorem{remark}[theorem]{Remark}

\theoremstyle{definition}
\newtheorem{definition}[theorem]{Definition}

\newcommand{\bZ}{\mathbb{Z}}      

\usepackage{cleveref}
\usepackage[backend=bibtex,maxbibnames=99,firstinits=true,style=ieee,sorting=nty,style=numeric]{biblatex}
\DeclareFieldFormat{postnote}{#1}
\allowdisplaybreaks
\begin{document}

\title{Cubiquitous Lattices and Branched Covers bounding rational balls}

\author[Choi]{Erica Choi}
\author[Saglam]{Nur Saglam}
\author[Simone]{Jonathan Simone}
\author[Stuopis]{Katerina Stuopis}
\author[Zhou]{Hugo Zhou}

\maketitle 

\begin{abstract} 
In \cite{greeneowens}, Greene and Owens explore \textit{cubiquitous} lattices as an obstruction to rational homology 3-spheres bounding rational homology 4-balls. 
The purpose of this article is to better understand which sublattices of $\Z^n$ are cubiquitous with the aim of effectively using their cubiquity obstruction.
We develop a geometric obstruction (called the Wu obstruction) to cubiquity and use it as tool to completely classify which sublattices with orthogonal bases are cubiquitous. 
We then apply this result to the double branched covers of alternating connected sums of torus links. Finally, we explore how the Wu obstruction can be used in conjunction with \textit{contractions} to obstruct the cubiquity of infinite families of lattices.
\end{abstract}

\section{Introduction}
A rational homology 3-sphere is a closed 3-manifold $Y$ satisfying $H_*(Y;\Q)\cong H_*(S^3;\Q)$; a rational homology 4-ball is a 4-manifold $X$ with boundary satisfying $H_*(X;\Q)\cong H_*(B^4;\Q)$. It is a straightforward exercise to show that the boundary of a rational homology 4-ball is a rational homology 3-sphere. The converse question---does a given rational homology 3-sphere bound a rational homology 4-ball---is a well-known problem in low-dimensional topology (cf. Problem 4.5 in \cite{kirbyproblemlist}).
Although there are many (partial) classifications of infinite families of rational homology 3-spheres bounding rational homology 4-balls (e.g. \cite{liscalensspaces}, \cite{lecuonamontesinosknots}, \cite{acetogolla}, \cite{akbulutlarson}, \cite{fickle}, \cite{cassonharer}, \cite{simonecircles}), the question at large is wide open.

A rather successful method of obstructing rational homology 3-spheres from bounding rational homology 4-balls comes in the form of lattice embeddings, following from Donaldson's Diagonalization Theorem \cite{donaldson}. A \textit{lattice} is a pair $(A,Q)$, where $A\cong \Z^n$ and $Q$ is a nondegenerate symmetric bilinear form on $A$. Given two lattices $(A,Q)$ and $(B,R)$, a \textit{lattice embedding} $i:(A,Q)\to (B,R)$ is an injective linear map $i:A\to B$ satisfying $R(i(x),i(y))=Q(x,y)$ for all $x,y\in A$.

\begin{obstruction} 
Let $Y$ be a rational homology 3-sphere that bounds a 4-manifold $X$ with positive definite intersection form $Q_X$. Suppose $Y$ also bounds a rational homology 4-ball. Then there exists a \textit{lattice embedding} $$i:(H_2(X), Q_X)\to (\Z^{n}, I),$$ where $n=\text{rank}(H_2(X))$ and $I$ is the standard dot product.
\label{ob:donaldson}
\end{obstruction}

In \cite{greenejabuka} (reinterpreted in \cite{greeneowens}), it is shown that Heegaard Floer homology $d$-invariants can often be used to extract additional geometric information from the lattice embedding, which provides a more robust obstruction.
This additional geometric information comes in the form of \textit{cubiquity}. 
We say that a lattice is \emph{cubiquitous} if it admits a lattice embedding into some $\bZ^n$ such that its image $\Lambda$ hits every unit cube in $\Z^n$; namely,
\[
\Lambda \cap (x+ \{0,1\}^n) \neq 0
\]
for every $x \in \Z^n.$ Moreover, we say that the sublattice $\Lambda$ of $\Z^n$ is cubiquitous.

\begin{obstruction} 
Let $Y$, $X$, and $i$ be as in Obstuction \ref{ob:donaldson}. If $X$ is  \textit{sharp}\footnote{We will not provide a precise definition of \textit{sharp} in this article, as all manifolds under consideration will automatically be sharp; see, e.g. \cite{greeneowens} for the precise definition of sharp.}, then $i$ is cubiquitous.
\label{ob:owens}
\end{obstruction}

Cubiquity provides a stronger obstruction (Obstruction \ref{ob:owens}) than the existence of a lattice embedding (Obstruction \ref{ob:donaldson}). 
For this reason, this article aims to better understand the question ``which sublattices of $\Z^n$ are cubiquitous?" We assume throughout that every sublattice $\Lambda$ is \textit{full-rank}; that is, $\text{rank}(\Lambda)=\text{rank}(\Z^n)=n$. It is clear that sublattices that are not full-rank cannot be cubiquitous. In \cite{greeneowens}, Greene-Owens completely classify which full-rank sublattices of $\Z^n$ with ``large bases" are cubiquitous.

\begin{theorem}[\cite{greeneowens}] 
Let $\Lambda\subset\Z^n$ be a full-rank sublattice and let $B$ be a matrix whose columns form a basis for $\Lambda$.
\begin{itemize}
    \item if $\det B>2^n$, then $\Lambda$ is not cubiquitous.
    \item if $\det B=2^n$, then $\Lambda$ is cubiquitous if and only if $\Lambda$ has a Haj{\'o}s basis\footnote{A Haj{\'o}s basis is one formed from the columns of a lower triangular matrix with $2$'s on the diagonal and $0$s and $1$s below the diagonal.}.
\end{itemize}
\label{thm:greeneowens1}
\end{theorem}

In the present article, we explore necessary conditions for a sublattice to be cubiquitous, without restriction on the size of the basis (measured via determinant). The main obstructive tool we will develop applies to a large class of sublattices having so-called \textit{non-acute} bases; we defer the precise definition of \textit{non-acute} to Section \ref{sec:subsets}. This obstruction relies on the \textit{Wu element} associated to the basis of a sublattice; given a sublattice with basis $S=\{v_1,\ldots,v_n\}$, the Wu element of $S$ is $W=\sum_{i=1}^nv_i.$ In the following, $\{e_1,\ldots,e_n\}$ denotes the standard orthogonal basis for $\Z^n$

\begin{restatable}[Wu Obstruction]{proposition}{Wuobstruction}
Let $\Lambda\subset\bZ^n$ be a sublattice with basis $S=\{v_1,\ldots,v_n\}$ which is non-acute. Let $k_1,\ldots,k_n$ be the unique integers satisfying $W=\sum_{i=1}^nk_ie_i$ and set $R_o:=\{i\,:\,k_i\equiv 1\pmod 2\}$. If $$\sum_{i=1}^nk_i^2>4n-3|R_o|,$$ then $\Lambda$ is not cubiquitous. 
\label{prop:Wuobstruction}
\end{restatable}

We use this obstruction to explore the cubiquity of full-rank sublattices with orthogonal bases, which we call \textit{orthogonal sublattices}.

\begin{theorem} An orthogonal sublattice of $\Z^n$ is cubiquitous if and only if it admits a basis $\mathcal{B}=\{b_1,\ldots,b_n\}$ such that the matrix $B=\begin{bmatrix} b_1&\cdots&b_n\end{bmatrix}$ is a block matrix with blocks of the form $\begin{bmatrix}1\end{bmatrix}$, $\begin{bmatrix}2\end{bmatrix}$, and 
$\begin{bmatrix}1&-1\\1&1\end{bmatrix}$ up to changing the basis of $\Z^n$.
\label{thm:main}
\end{theorem}

\subsection{Application to Double Branched Covers}

In \cite{greeneowens}, cubiquitous lattice embeddings are studied in relation to a large family of rational homology 3-spheres---those
arising as the double branched covers over non-split alternating links in $S^3$ (note that the double branched cover along a split link is not a rational homology 3-sphere); we use $\Sigma(L)$ to denote the double cover of $S^3$ branched along $L$. 
Given such a link $L$, there is a well-defined positive definite lattice $\Lambda(L)$ induced by $L$, which is unique up to isometry  (see the discussion above Theorem 1 in \cite{greeneowens} for details). 
The following result can be viewed as a special case of Obstruction \ref{ob:owens}.

\begin{theorem}[\cite{greeneowens}] Let $L$ be a non-split alternating link. If $\Sigma(L)$ bounds a rational homology ball, then $\Lambda(L)$ is cubiquitous.
\label{thm:greeneowens}
\end{theorem}

It is further conjectured by Greene-Owens in \cite{greeneowens}  that in fact, cubiquity is a sufficient condition for $\Sigma(L)$ bounding a rational homology 4-ball.

\begin{conjecture}[\cite{greeneowens}]\label{conj:main}
For a non-split alternating link $L,$ $\Sigma(L)$ bounds a rational ball $\Longleftrightarrow$ $\Lambda(L)$ is cubiquitous.
 \end{conjecture}

In \cite{greeneowens}, Greene-Owens resolve Conjecture \ref{conj:main} for non-split alternating links with $|H_1(\Sigma(L))|=4^n$; this follows from the second part of Theorem \ref{thm:greeneowens1} (note that if $\Sigma(L)$ bounds a rational homology 4-ball, then by \cite{cassongordon} the image of the lattice embedding given by Obstruction \ref{ob:donaldson} has a basis whose determinant is precisely $\sqrt{|H_1(\Sigma(L))|}$ ). 
Theorem \ref{thm:main} provides us with a similar result for alternating connected sums of torus links.

\begin{theorem}\label{thm:maintop}
    Conjecture \ref{conj:main} holds for alternating connected sums of torus links of the same sign.
 \end{theorem}

 \begin{remark}
     Given a family of torus links, there are different ways to perform the connected sum, yielding non-isotopic links as a result. However,  it turns out that the isomorphism type of the lattice induced  by the connected sum does not depend on the way one performs the connected sum. See the discussion in Section \ref{sec:toruslinks}.
 \end{remark}

 \begin{remark}
     Theorem \ref{thm:maintop} is already implicitly known in the literature. It follows from \cite{lisca-sumsoflensspaces}, \cite{greene}, and calculations of Heegaard Floer homology $d$-invariants of connected sums of lens spaces. However, Theorem \ref{thm:main} provides more information. The results in the literature show that $\Sigma(L)$---where $L$ is an alternating connected sum of torus links of the same sign---bounds a rational homology 4-ball if and only if the lattice embedding $\Lambda(L)\to \Z^n$ coming from Obstruction \ref{ob:donaldson} is cubiquitous. However, $\Lambda(L)$ can admit many lattice embeddings into $\Z^n$ and it is not clear which of these embeddings comes from Obstruction \ref{ob:donaldson}; hence Obstruction \ref{ob:owens} cannot tell us precisely which lattice embedding is cubiquitous. Theorem \ref{thm:main}, however, tells us exactly which lattice embeddings are cubiquitous.
 \end{remark}

\subsection{Contractions}
The notion of \textit{contractions} (defined in Section \ref{sec:contractions}) has been paramount in understanding the existence of certain kinds of sublattices (see, e.g., \cite{liscalensspaces}, \cite{lisca-sumsoflensspaces}, \cite{simone}), which has led to a full understanding of which lens spaces bound rational homology 4-balls, for example. In particular, contractions often allow one to create infinite families of sublattices.

\begin{restatable}{lemma}{contractions}
Contractions of non-acute sublattices preserve the Wu obstruction given in Proposition \ref{prop:Wuobstruction}.
\label{lem:preserve}
\end{restatable}

This straight-forward calculation allows one to obstruct the cubiquity of infinite families of sublattices, which in turn, can aid one in obstructing the existence of rational homology 4-balls bounded by infinite families of rational homology 3-spheres. In Section \ref{sec:contractions}, we will see how this strategy applies to sublattices with so-called \textit{good} bases (defined by Lisca in \cite{liscalensspaces}), which has implications on lens spaces bounding rational homology 4-balls.

\subsection*{Organization}  In Section \ref{sec:subsets}, we will consider a particular family of sublattices having bases called \textit{non-acute subsets} and prove Proposition \ref{prop:Wuobstruction}.
The proof of Theorem \ref{thm:main} will be carried out in Sections \ref{sec:p4} and \ref{sec:proof}. In Section \ref{sec:toruslinks}, we will outline the connection between orthogonal sublattices and alternating connected sums of torus links, and prove Theorem \ref{thm:maintop}. Finally, in Section \ref{sec:contractions}, we will prove Lemma \ref{lem:preserve} and apply it to so-called \textit{good subsets}.

\subsection*{Acknowledgements}
This work originated in an REU at Georgia Tech that was supported by the NSF grants 1745583, 1851843, 2244427 and the Georgia Tech
College of Sciences. The authors would also like to thank the anonymous referees for their thorough review.

\section{Non-acute subsets}\label{sec:subsets}

For vectors $v,w\in \mathbb{Z}^n$, we use $\langle v,w\rangle$ to denote the standard positive-definite dot product of $v$ and $w$. We denote the standard orthogonal basis of $\Z^n$ by $\{e_1,\ldots,e_n\}$. Following \cite{liscalensspaces}, we now define some useful notation.

\begin{definition} Let $S=\{v_1,\ldots,v_n\}\subset\Z^n$ be a subset with $\langle v_i, v_i\rangle=a_i$. We define the following notation. $$I(S):=\displaystyle\sum_{i=1}^n(a_i -3)$$
	$$E^S_j:=\{i\text{ }:\text{ } \langle v_i, e_j\rangle\neq0\}$$
	$$ V^S_{i}:=\{j\text{ }:\text{ } \langle v_i, e_j\rangle \neq 0\}$$
    $$\mathcal{P}_i^S:=\{j\text{ }:\text{ } |E_j^S|=i\}$$
	$$p_i(S):=|\mathcal{P}_i^S|$$
 
\noindent Note that the sizes of the above sets can be interpreted using the matrix $\begin{bmatrix}v_1&v_2&\cdots& v_n\end{bmatrix}$; indeed, $|E_j^S|$ is the number of columns whose $j$-th entry is non-zero, $|V^S_{i}|$ is the number of rows whose $i$-th entry is non-zero, and $|\mathcal{P}_i^S|$ is the number of rows with exactly $i$ non-zero entries. In some cases we will drop the superscript $S$ from the above notation if the subset being considered is understood.\end{definition}

The following formula is a strengthening of Lemma 2.5 in \cite{liscalensspaces}.

\begin{lemma} Given any subset $S=\{v_1,\ldots,v_n\}\subset\Z^n$,
    $$2p_1(S) + p_2(S) + I(S) = \sum_{j=4}^n(j-3)p_j(S) + \sum_{\substack{s, i\\ \langle v_s, e_i \rangle \neq 0}}(\langle v_s, e_i \rangle ^2 - 1)$$
\label{lem:equation}
\end{lemma}
\begin{proof}
\begin{align*}
    \sum_{j=4}^n&(j-3)p_j(S) + \sum_{\substack{s, i\\ \langle v_s, e_i \rangle \neq 0}}(\langle v_s, e_i \rangle ^2 - 1)\\
    &= \sum_{j=4}^njp_j(S) - 3\sum_{j = 4}^np_j(S) + \sum_{\substack{s,i\\ \langle v_s, e_i \rangle \neq 0}}(\langle v_s, e_i \rangle ^2 - 1) \\
    &= \sum_{j=4}^njp_j(S) - 3\sum_{j = 1}^np_j(S) + 3p_1(S) + 3p_2(S) + 3p_3(S)+ \sum_{\substack{s, i\\ \langle v_s, e_i \rangle \neq 0}}(\langle v_s, e_i \rangle ^2 - 1)\\   
    &= \sum_{j=1}^njp_j(S) - 3n + 2p_1(S) + p_2(S) + \sum_{\substack{s, i\\ \langle v_s, e_i \rangle \neq 0}}(\langle v_s, e_i \rangle ^2 - 1)\\
    &= \sum_{j=1}^njp_j(S) - 3n + 2p_1(S) + p_2(S) + \sum_{\substack{s, i\\ \langle v_s, e_i \rangle \neq 0}}\langle v_s, e_i \rangle ^2 - \sum_{j=1}^njp_j(S)\\
    &= - 3n + 2p_1(S) + p_2(S) + \sum_{i=1}^na_i\\
    &= 2p_1(S) + p_2(S) + \sum_{i=1}^n(a_i-3)\\
    &= 2p_1(S) + p_2(S) + I(S)
\end{align*}
\end{proof}

\begin{definition}
    A set $S = \{v_1, ..., v_n\}$ is \textit{non-acute} if $a_i=\langle v_i,v_i\rangle\geq 1$ for all $i$, $\langle v_i,v_j\rangle\le0$ for all $i\neq j$, and
    $$a_i\ge -\sum_{j\neq i}\langle v_j,v_i\rangle\text{ for all }i.$$
    Furthermore, $S$ is called \textit{orthogonal} if we impose the condition that $\langle v_i,v_j\rangle=0$ for all $i\neq j$.    
\end{definition}

\begin{remark}
    A lattice admits a non-acute basis if and only if it admits an \textit{obtuse superbase}, which is a well-known concept in the literature. Hence, the results in this paper shed light on the cubiquity of lattices with obtuse superbases.
\end{remark}

\begin{definition}
    Given a subset $S=\{v_1,\ldots,v_n\}\subset\Z^n$ with Wu element $W:=\displaystyle\sum_{i=1}^nv_i=\displaystyle\sum_{i=1}^n k_ie_i$, we define the following sets
    \begin{equation*}
        \begin{split}
            \mathcal{O}^S&=\{i\,|\,k_i=0\}\\
            R_o^S&=\{i\,|\,k_i\equiv 1\pmod 2\}\\
            R_e^S&=\{i\,|\,k_i\equiv 0\pmod 2, k_i\neq 0\}
        \end{split}
    \end{equation*}
    We will usually drop the superscript notation when the subset being considered is understood.
\end{definition}

The following cubiquity obstruction comes from examining $S$ and is a generalization of Theorem 6.4 in \cite{brejevssimone}.

\Wuobstruction*

\begin{proof}
Let $z=\frac{1}{2}W$. For a unit cube $C\subset\Z^n$, let $\overline{C}\subset\R^n$ denote the solid unit cube containing $C$. Let $K$ be the union of all unit cubes $C$ such that $z\in\overline{C}$. 
We will show that $\Lambda\cap K=\emptyset$.
Notice that $x\in K$ if and only if $W-x\in K$; similarly, $x\in \Lambda$ if and only if $W-x\in \Lambda$. Hence $x\in \Lambda\cap K$ if and only if $W-x\in \Lambda\cap K$.

Let $x\in \Lambda$. Then $x$ is of the form $x=\sum_{i=1}^nx_iv_i$ and $W-x=\sum_{i=1}^n(1-x_i)v_i$. Hence the dot product of these vectors is
\begin{equation*}
    \begin{split}
        \langle x,W-x\rangle &=\sum_{i=1}^nx_i(1-x_i)a_i+\sum_{i<j}(x_i-2x_ix_j+x_j)\langle v_i,v_j\rangle\\
        &=\Big(\sum_{i=1}^nx_i(1-x_i)(a_i+\sum_{j\neq i}\langle v_i, v_j\rangle)\Big)+\sum_{i<j}(x_i-x_j)^2\langle v_i,v_j\rangle\\
        &\le 0.
    \end{split}
\end{equation*}
The last inequality follows from the fact that for non-acute subsets, $\langle v_i,v_j\rangle\le0$ for all $i\neq j$ and $a_i\ge -\sum_{j\neq i}\langle v_j,v_i\rangle$ for all $i$.

Next let $x\in K$. Then $x$ is of the form $$x=\sum_{i\in R_o}\frac{k_i+\epsilon_i}{2}e_i+\sum_{i\in R_e}\frac{k_i+2\delta_i}{2}e_i+\sum_{i\in \mathcal{O}}\delta_ie_i,$$
where $\epsilon_i\in\{-1,1\}$ for all $i\in R_o$ and $\delta_i\in\{-1,0,1\}$ for all $i\in R_e\cup\mathcal{O}$. Hence $W-x$ is of the form 
    $$W-x=\sum_{i\in R_o}\frac{k_i-\epsilon_i}{2}e_i+\sum_{i\in R_e}\frac{k_i-2\delta_i}{2}e_i-\sum_{i\in \mathcal{O}}\delta_ie_i.$$
Since $\sum_{i=1}^n k_i^2>4n-3|R_0|$, we have that the dot product between these two vectors is
\begin{equation*}
    \begin{split}
        \langle x,W-x\rangle&=\sum_{i\in R_o}\frac{k_i^2-1}{4}+\sum_{i\in R_e}\frac{k_i^2-4}{4}-|\mathcal{O}|\\
        &=\sum_{i=1}^n \frac{k_i^2}{4}-\frac{|R_o|+4|R_e|+4|\mathcal{O}|}{4}\\
        &=\sum_{i=1}^n \frac{k_i^2}{4}-\frac{4n-3|R_o|}{4}\\
        &>0.
    \end{split}
\end{equation*}
The result follows.   
\end{proof}

For the sake of exposition, if the lattice generated by a subset $S=\{v_1,\ldots,v_n\}$ is cubiqutious, then we call $S$ \textit{cubiquitous}; otherwise we say that $S$ is not cubiquitous. 

\begin{corollary}
Let $S\subset\mathbb{Z}^n$ be an orthogonal subset with Wu element $W=\sum_{i=1}^nk_ie_i$. If $$I(S)>n-3|R_o|,$$ then $S$ is not cubiquitous.
    \label{cor:wuobstruction}
\end{corollary}

\begin{proof}
    Since $S$ is orthogonal, we can compute $\langle W, W\rangle$ in two simple ways. First, since $W=\sum_{i=1}^nk_ie_i$, $\langle W, W\rangle =\sum_{i=1}^nk_i^2$. On the other hand, since $W=\sum_{i=1}^nv_i$, $\langle W, W\rangle = \sum_{i=1}^n a_i$. Thus the inequality 
    $\sum_{i=1}^nk_i^2>4n-3|R_o|$ can be rewritten as $\sum_{i=1}^na_i>4n-3|R_o|$, or $I(S)>n-3|R_o|$. Now by Proposition \ref{prop:Wuobstruction}, the result follows.
\end{proof}

We end this section with a result about direct sums of (cubiquitous) lattices. 

\begin{lemma} 
Let $\Lambda_1\subset \bZ^m$ and $\Lambda_2\subset\bZ^n$ be full-rank sublattices and let $\Lambda=\Lambda_1\oplus\Lambda_2\subset\bZ^k$, where $k=m+n$. Then $\Lambda$ is cubiquitous if and only if $\Lambda_1$ and $\Lambda_2$ are cubiquitous.
\label{lem:block}
\end{lemma}

\begin{proof}
Suppose that $\Lambda_1$ and $\Lambda_2$ are cubiquitous. 
Let $C$ be a unit cube in $\bZ^k$ and let $C_1$ and $C_2$ denote the projections of $C$ to $\bZ^m$ and $\bZ^n$, respectively.
Then there exists $x_1\in C_1\cap \Lambda_1$ and $x_2\in C_2\cap \Lambda_2$. Since $\Lambda=\Lambda_1\oplus\Lambda_2$, it follows that $x=(x_1,x_2)\in C\cap \Lambda$; hence $\Lambda$ is cubiquitous.

For the converse, assume that either $\Lambda_1$ or $\Lambda_2$ is not cubiquitous; without loss of generality assume the former. Then there is a unit cube $C_1\subset \bZ^m$ that does not intersect $\Lambda_1$. Let $C\subset\bZ^k$ be a unit cube that projects onto $C_1$ under the projection $\bZ^k\to \bZ^m$. Let $C_2$  denote the projection of $C$ to $\bZ^n$; then $C=C_1\oplus C_2$.
Pick a vector $x\in C$. Let $x_1\in C_1$ and $x_2\in C_2$ be the unique vectors such that $x=(x_1,x_2)\in C_1\oplus C_2=C$. Since $C_1\cap \Lambda_1=\emptyset$, we have that $x_1\not\in\Lambda_1$. It now follows that since $\Lambda=\Lambda_1\oplus\Lambda_2$, $x\not\in\Lambda$. Hence $C\cap \Lambda=\emptyset$ and so $\Lambda$ is not cubiquitous.
\end{proof}

\section{Orthogonal Subsets with $p_4(S)=n$}\label{sec:p4}
The proof of Theorem \ref{thm:main} will rely on the Wu obstruction and, in particular, Corollary \ref{cor:wuobstruction}. However, this obstruction does not apply for subsets $S$ satisfying $I(S)=n-3|R_o|$. We will come across such subsets in the proof of Theorem \ref{thm:main}; in particular, we will need an understanding of orthogonal subsets $S$ satisfying $p_4(S)=n$. 

\begin{theorem} Let $S=\{v_1,\ldots,v_n\}\subset\Z^n$ be an orthogonal subset with $p_4(S)=I(S)=n$ and let $B=\begin{bmatrix}v_1&\cdots &v_n\end{bmatrix}$. Then up to a change of basis of $\Z^n$ and reordering or negating columns of $B$, $B$ is a block diagonal matrix such that each block is either of the form

$$\begin{bmatrix}1&1&1&1&0&0&0&0\\1&-1&-1&-1&0&0&0&0\\0&1&-1&-1&1&0&0&0\\0&-1&1&1&1&0&0&0\\0&0&1&-1&0&1&1&0\\0&0&-1&1&0&1&-1&0\\0&0&1&-1&0&0&-1&1\\0&0&-1&1&0&0&1&1\end{bmatrix}\text{ or }\begin{bmatrix}1&1&1&1&0&0&0&0\\1&-1&-1&-1&0&0&0&0\\0&-1&1&0&1&1&0&0\\0&1&-1&0&1&-1&0&0\\0&1&0&-1&0&1&1&0\\0&-1&0&1&0&-1&1&0\\0&0&-1&1&0&1&0&1\\0&0&1&-1&0&-1&0&1\end{bmatrix},$$
or has the property that each column has precisely four nonzero entries, each of which are either $1$ or $-1$.
    \label{thm:p4=n}
\end{theorem}

\begin{remark}
    Note that although reordering or negating columns of $B$ changes the subset $S$, it does not change the sublattice $\Lambda$ generated by $S$. Indeed, these operations simply provide a different basis for $\Lambda$. Since our ultimate goal is to classify cubiquitious orthogonal sublattices, there is no harm in considering subsets up to these operations.
\end{remark}

\begin{remark}
    Determining the explicit blocks listed in the statement of Theorem \ref{thm:p4=n} requires case-by-case lattice analysis. One can apply the same strategy to  explicitly determine the blocks whose columns have precisely four nonzero entries (for example, one can find an infinite family of such blocks). However, for the purposes of this paper, we do not need to know these blocks explicitly; in determining cubiquity, it will be enough to know that the columns have exactly four nonzero entries.
\end{remark}

Notice that Corollary \ref{cor:wuobstruction} cannot be applied to the subsets listed in Theorem \ref{thm:p4=n} since for those subsets we have $|R_o|=0$ and $I(S)=n$. Hence to determine cubiquity, we cannot use Corollary \ref{cor:wuobstruction}. We will instead appeal to Theorem \ref{thm:greeneowens1} as well as elementary linear algebra to prove the following.

\begin{theorem}
    If $S=\{v_1,\ldots,v_n\}\subset\Z^n$ is an orthogonal subset with $p_4(S)=n$, then $S$ is not cubiquitous.
    \label{thm:p4notcb}
\end{theorem}
    
To prove these results, we first record some elementary calculations.

\subsection{Determinant Calculations}
Let $a,b,c,d\ge 1$ be integers and consider the matrix
$$M=\begin{bmatrix}a&-1&-1&-1\\-1&b&-1&-1\\-1&-1&c&-1\\-1&-1&-1&d\end{bmatrix}.$$
Then
    \begin{equation}
        \det M=-2a-ab-2b-ac-bc-2c-ad-bd+abcd-cd-2d-3.
        \label{eq:det}
    \end{equation}

\begin{lemma} If $\det M=0$, then $\min\{a,b,c,d\}\le 3$.
    \label{lem:det}
\end{lemma}

\begin{proof}
    By symmetry, we may assume that $1\le a\le b\le c\le d$. We must show that $a\le 3$.
    Assume that $a\ge 4$; then $b,c,d\ge4$. Setting equation (\ref{eq:det}) to $0$ and solving for $d$ yields 
    \begin{equation}
    d=\frac{ab+ac+bc+2(a+b+c)+3}{abc-(a+b+c+2)}
        \label{eq:d}
    \end{equation}
    Note that since $a,b,c\ge4$, we have that $abc>a+b+c+2$ and so equation (\ref{eq:d}) is defined.  
    Since the derivative of $d$ with respect to $a$ is $\displaystyle-\frac{(b+1)^2(c+1)^2}{(abc-(a+b+c+2))^2}$, $d$ is decreasing as a function of $a$. Similar calculations show that $d$ is also decreasing as a function of $b$ and as a function of $c$. Subject to the contraint $a,b,c\ge4$, it follows that $d$ is maximized when $a=b=c=4$; hence $d\le \frac{3}{2}$, which is a contradiction.
\end{proof}

\begin{lemma}
    If $\det M=0$ and $\min\{a,b,c,d\}=3$, then $a=b=c=d=3$.
    \label{lem:deta=3}
\end{lemma}
    \begin{proof}
    By symmetry, we may assume that $1\le a\le b\le c\le d$.
    Since $a=3$, equation (\ref{eq:d}) becomes 
    $$d=\frac{bc+5b+5c+9}{3bc-b-c-5}.$$
    Let $b\ge 4$.
    As in the proof of Lemma \ref{lem:det}, $d$ is decreasing as a function of both $b$ and $c$; hence $d$ is maximized when $b=c=4$, which implies $d\le \frac{13}{7}<2<c$, which is a contradiction. Hence $b=3$ and $d=\displaystyle\frac{c+3}{c-1}$. Now since $d$ is decreasing as a function of $c$ and $d\ge c\ge 3$, it follows that $c=d=3$.
    \end{proof}

\begin{lemma}
    Let $\det M=0$ and $\min\{a,b,c,d\}=1$. 
   Assume, by symmetry, that $1=a\le b\le c\le d$. Then $b\le 5$. Moreover, 
    \begin{itemize}
        \item if $b=3$ and $c\ge7$, then $c=d=7$;
        \item if $b=4$, then $c=4$ and $d=9$;
        \item if $b=5$, then $c=d=5$.
    \end{itemize}
    \label{lem:deta=1}
\end{lemma}

    \begin{proof}
    Since $a=1$, equation (\ref{eq:d}) becomes 
    $$d=\frac{bc+3b+3c+5}{bc-b-c-3}.$$
    If $b\ge 6$, then as in the proof of Lemma \ref{lem:det}, $d$ is decreasing as a function of both $b$ and $c$; hence $d$ is maximized when $b=c=6$, which implies $d\le \frac{17}{4}<b\le c$, which is a contradiction.

    If $b=3$ and $c\ge7$, then since $d$ is decreasing in $c$, we have $d=\displaystyle\frac{3c+7}{c-3}\le 7$; since $d\ge c\ge7$, it follows that $d=c=7$. 
    If $b=5$, then equation (\ref{eq:d}) becomes $\displaystyle d=\frac{2c+5}{c-2}$; since $d\ge c\ge b=5$ and $d$ is decreasing in $c$, we necessarily have that $c=d=5$.
    Finally, let $b=4$. Then equation (\ref{eq:d}) becomes $\displaystyle d=\frac{7c+17}{3c-7}$. If $c=5$ or $c=6$, then $d\not\in\Z$, which is a contradiction. If $c\ge7$, then since $d$ is decreasing in $c$, we have $d\le \frac{33}{7}<c$, which is a contradiction. Hence $c\le 4$. But since $c\ge b=4$, we have $c=4$ and $d=9$.
    \end{proof}

\begin{lemma}
    If $\det M=0$ and two of the diagonals are 7, then the others are 1 and 3. 
    \label{lem:detab77}
\end{lemma}

\begin{proof}
    Up to symmetry, we may assume $a=b=7$. Then equation (\ref{eq:d}) becomes $\displaystyle d=\frac{c+5}{3c-1}$. If $c\ge4$, then since $d$ is decreasing in $c$ and $c\ge 4$, we have $d\le \frac{9}{11}$, which contradicts $d\ge 1$. If $c=3$, then $d=1$, as desired. If $c=2$, then $d\not\in \Z$, which is a contradiction. If $c=1$, then $d=3$, as desired. 
\end{proof}

\begin{lemma} If $\det M=0$ and two of the diagonals are $5$ and $5$ then the others are either: $1$ and $5$; or $2$ and $2$.
    \label{lem:55}
\end{lemma}

\begin{proof} Up to symmetry, assume $a=b=5$ and $c\le d$.
Equation (\ref{eq:d}) reduces to $d=\displaystyle\frac{c+4}{2c-1}$. First note that if $c=1$, then $d=5$. Now assume that $c\ge 3$; since $d$ is decreasing in $c$, it follows that $d\le \frac{7}{5}<c$, which is a contradiction. If follows that $c=2$ and $d=2$.
\end{proof}

\subsection{Proof of Theorem \ref{thm:p4=n}}

Let $S=\{v_1,\ldots,v_n\}\subset\Z^n$ be orthogonal with $p_4(S)=n$ and $I(S)=n$. Since $p_4(S)=n$, we have $p_j(S)=0$ for all $j\neq 4$. It follows from Lemma \ref{lem:equation} that $|\langle v_\alpha, e_i\rangle|\le 1$ for all $\alpha, i$. 

Fix $i$ and let $E_i=\{s,t,u,w\}$. Define $v_\alpha'=v_\alpha-\langle v_\alpha,e_i\rangle e_i$ and
let $$S_i=(S\setminus\{v_s, v_t, v_u, v_w\})\cup\{v_s', v_t', v_u', v_w'\}\subset\Z^{n-1}$$
be the subset obtained by projecting the vectors of $S$ onto $\Z^{n-1}=\langle e_1,\cdots,e_{i-1},e_{i+1},\ldots,e_n\rangle$. Let $C_i$ be the matrix whose columns are the vectors of $S_i$, ordered so that the first four columns of $C_i$ are $v_s', v_t', v_u',$ and $v_w'$. Then $Q_i=C_i^TC_i$ is a block diagonal matrix whose first block is the $4\times 4$ matrix $$M_i=\begin{bmatrix}a_s-1&-1&-1&-1\\-1&a_t-1&-1&-1\\-1&-1&a_u-1&-1\\-1&-1&-1&a_w-1\end{bmatrix}.$$ and whose remaining blocks are of the form $[a_j]$ for $j\not\in\{s,t,u,w\}$.
Since the columns of $C_i$ are linearly dependent (as $C_i$ is an $(n-1)\times n$ matrix), we have that $\det Q_i=0$. Since $a_j>0$ for all $j$, it follows that $\det M_i=0$. We will appeal to this fact multiple times throughout this section.

Let $B'$ be a block of $B$. If $B'$ satisfies the property that each column has exactly four nonzero entries, then we are done. Assume otherwise that $B'$ has a column that does not have exactly four nonzero entries. Up to reordering, let $B'=\begin{bmatrix} v_1&\cdots&v_k\end{bmatrix}$ and assume that $a_1\neq 4$. We may further assume (up to change of basis and negation) that $E_1=\{1,2,3,4\}$ and $\langle v_\alpha,e_1\rangle=1$ for $1\le \alpha\le 4$. Since $\langle v_1,v_2\rangle =0$, we may assume up to change of basis that $v_1=e_1+e_2+f$ and $v_2=e_1-e_2+g$, where $f,g\in\Z^{n-2}$ are vectors satisfying $\langle f, g\rangle=0$.
    Consider the matrix $M_1$ and set $a=a_1-1$, $b=a_2-1$, $c=a_3-1$, and $d=a_4-1$.
    Since $\det M_1=0$ and $a\neq 3$, it follows from Lemma \ref{lem:deta=3} that $\min\{a,b,c,d\}\neq3$. Hence we may assume that $\min\{a,b,c,d\}\le 2$. 
    Furthermore, up to reordering once again, we may assume that $a=\min\{a,b,c,d\}$ and $1\le a\le b\le c\le d$.
    Hence $a\le 2$. 
    We now show that $a\neq 2$.

    Suppose $a=2$; then up to a change of basis, we have $f=e_3$ and so $v_1=e_1+e_2+e_3$. 
    Since $\langle v_1,v_2\rangle=0$, it is clear that $3\not\in V_2$.
    Since $\langle v_1,v_\alpha\rangle=0$ for $\alpha=3,4$, we must have that: $2\in V_3$ or $3\in V_3$; and $2\in V_4$ or $3\in V_4$. Further note that since $|\langle v_\alpha, e_i\rangle|\le 1$ for all integers $i$ and $\alpha$, we cannot have $2,3\in V_3$ or $2,3\in V_4$. If $2\in V_3\cap V_4$ (and so $3\not\in V_3\cup V_4$), then since $E_1=E_2=\{1,2,3,4\}$ and $S$ is orthogonal, it follows that $E_3=\{1\}$, which is a contradiction. Now suppose $2\in V_3$ and $2\notin V_4$ (or similarly, $2\in V_4$ and $2\notin V_3$). Then $E_2\supset\{1,2,3,s\}$ and $E_3\supset\{1,4,t,u\}$ for some integers $s,t,u$. Since $3\in V_t$, $\langle v_t,v_1\rangle=0$, and $E_1=\{1,2,3,4\}$, it follows that $2\in V_t$ and so $t=s$. Now since $u\neq t$, it follows that $\langle v_u,v_1\rangle\neq 0$, which is a contradiction. Finally suppose that $3\in V_2\cap V_4$ (and so $2\notin V_3\cup V_4$). Then an argument similar to the previous one provides a contradiction. Hence $a\neq 2$
   and so $a=1$.
   
    To finish the proof of Theorem \ref{thm:p4=n}, we will show that (up to change of basis) $B'$ must be one of the two matrices listed in the statement of Theorem \ref{thm:p4=n}. 
    Since $a=1$ and $\langle v_\alpha,v_1\rangle=0$ for $\alpha=3,4$, we have that $v_1=e_1+e_2$, $v_2=e_1-e_2+g$, $v_3=e_1-e_2+h$, and $v_4=e_1-e_2+k$, for some vectors $g,h,k\in\Z^{n-2}$ satisfying $\langle g,h\rangle=-2$, $\langle g,k\rangle=-2$, and $\langle h,k\rangle =-2$. Hence $E_1=E_2=\{1,2,3,4\}$. 
    It follows from Lemma \ref{lem:deta=1} (applied to $M_1$) that $b\le 5$. Since $\langle g,h\rangle=-2$, we have that $b+1=a_2\ge 4$. We have three cases to consider: $b=3,4,5$.\\
    
    \noindent\underline{$b=3$}. 
    Up to change of basis, we may assume $g=e_3-e_4$ so that $v_2=e_1-e_2+e_3-e_4$. Since $\langle g,h\rangle=\langle g,k\rangle=-2$, we necessarily have that $h=-e_3+e_4+h'$ and $k=-e_3+e_4+k'$ for some vectors $h', k'\in\Z^{n-4}$ satisfying $\langle h',k'\rangle=-4$. Up to change of basis, we may assume that $h'=e_5-e_6+e_7-e_8+h''$ and $k'=-e_5+e_6-e_7+e_8+k''$ for some vectors $h'', k''\in\Z^{n-8}$. 
    It follows that $c=a_3-1\ge 7$ and $d=a_4-1\ge 7$. Now by Lemma \ref{lem:deta=1}, we have that $c=d=7$ and so $h''=k''=0$.    
    Let $s$ be the integer satisfying $E_3=\{2,3,4,s\}$. Up to reordering, let $s=5$. Since $\langle v_5, v_2\rangle=0$ and $1,2\not\in V_5$, it follows that $v_5=e_3+e_4+l$ for some vector $l\in\Z^{n-4}$ satisfying $\langle l,v_\alpha\rangle=0$ for $\alpha=3,4$. Note that $E_3=E_4=\{2,3,4,5\}$.
    Consider the matrix $M_3$.
    Since $b=3$, $c=7$, $d=7$, and $\det M_3=0$, it follows from Lemma \ref{lem:detab77} that $a_5-1=1$; that is $l=0$.     
    Next, up to reordering, we may assume that $E_5=\{3,4,6,7\}$. 
    Consider the matrix $M_5$.
    Since $\det M_5=0$ and $c=d=7$, it follows from Lemma \ref{lem:detab77} that $a_6-1=1$ and $a_7-1=3$ or vice versa; up to reordering, we may  assume the former.
    Since $\langle v_4, v_6\rangle=0$, we have $V_6=\{5,i\}$, where $i\in\{6,7,8\}$. Up to change of basis, we may assume that $i=6$ and $v_6=e_5+e_6$. Since $\langle v_7,v_6\rangle=\langle v_7,v_4\rangle=0$, and $a_7=4$, we necessarily have that $v_7=e_5-e_6-e_7+e_8$ (up to negating $v_7$). Note that $E_5=E_6=\{3,4,6,7\}$. Finally, we set $E_7=\{3,4,7,8\}$. Since $\langle v_8, v_\alpha\rangle=0$ for all $\alpha,$ we necessarily have (up to negating $v_8$) that $v_8=e_7+e_8+m$ for some vector $m\in\Z^{n-8}$; consequently, $E_7=E_8=\{3,4,7,8\}$. 
    Considering the matrix $M_7$ and arguing as above,
    we see that $m=0$. 
    We have thus shown (up to change of basis of $\Z^n$, and reordering and negating vectors of $S$) that $B'$ is comprised of the following vectors:
    \begin{align*}
        v_1&=e_1+e_2\\
        v_2&=e_1-e_2+e_3-e_4\\
        v_3&=e_1-e_2-e_3+e_4+e_5-e_6+e_7-e_8\\
        v_4&=e_1-e_2-e_3+e_4-e_5+e_6-e_7+e_8\\
        v_5&=e_3+e_4\\
        v_6&=e_5+e_6\\
        v_7&=e_5-e_6-e_7+e_8\\
        v_8&=e_7+e_8
    \end{align*} 
    Hence $B'$ is the first matrix listed in Theorem \ref{thm:p4=n}.\\

    \noindent\underline{$b=4$}. By Lemma \ref{lem:deta=1}, we have $c=4$ and $d=9$.  
    Up to change of basis, we may assume $g=e_3+e_4+e_5$. Since $\langle g,h\rangle=-2$ and $c=4$, we may assume up to change of basis that $h=-e_3-e_4+e_6$. Since $\langle g,k\rangle=-2$, we also have either: $3\in V_4$ and $\langle v_4,e_3\rangle=-1$; or $4\in V_4$ and $\langle v_4,e_4\rangle=-1$. Assume the former so that $k=-e_3+k'$, where $k'\in\Z^{n-3}$. Up to reordering, we may assume that $E_3=\{2,3,4,5\}$. 
    Consider the matrix $M_3$.
    Since $\det M_3=0$, $b=c=4$, and $d=7$,
    we see from equation (\ref{eq:det}) that $a_5-1=1$. 
    Since $\langle v_5,v_2\rangle=0$, up to negation, either $v_5=e_3-e_4$ or $v_5=e_3-e_5$. In the latter case, we have $\langle v_5, v_3\rangle\neq 0$, which is a contradiction. In the former case, since $\langle v_5, v_4\rangle=0$, we must have $4\in V_4$; but since $\langle v_4, v_2\rangle=0$, it follows that $v_4=e_1-e_2-e_3-e_4+k''$ for some vector $k''\in\Z^{n-4}$. But then we have $\langle v_4, v_3\rangle \neq0$, which is a contradiction.\\

    \noindent\underline{$b=5$}. By Lemma \ref{lem:deta=1}, $c=d=5$. Up to change of basis, let $g=-e_3+e_4+e_5-e_6$. Since $\langle g,h\rangle=-2$, we may assume up to change of basis that $h=e_3-e_4+h'$ for some vector $h'\subset\Z^{n-4}$ satisfying $\langle g,h'\rangle=0$. 
    Suppose $5\in V_3$ (or similarly, $6\in V_3$); then since $\langle g,h'\rangle=0$, we have $6\in V_3$ and up to change of basis $h=e_3-e_4+e_5+e_6.$ Since $\langle v_4,v_2\rangle=0$ and $|\langle v_\alpha, e_i\rangle|\le 1$ for all $\alpha, i$, $|(V_2\cap V_4)\setminus\{1,2\}|$ is even; up to change of basis, we may assume that $3,4\in V_4$ and $v_4=e_1-e_2+e_3-e_4+k'$ for some vector $k'\in\Z^{n-4}$; but then $\langle v_4,v_3\rangle\neq 0$, which is a contradiction. Hence $5,6\not\in V_3$ and so we may assume up to change of basis that $v_3=e_1-e_2+e_3-e_4-e_7+e_8.$ 
    Now suppose $3\in V_4$ and let $E_3=\{2,3,4,s\}$.
    Consider the matrix $M_3$.
    Since $\det M_3=0$ and $b=c=d=5$, we necessarily have that $a_s-1=1$ by equation (\ref{eq:det}).
    Since $\langle v_s,v_2\rangle =\langle v_s,v_3\rangle=0$, we necessarily have that $v_s=e_3+e_4$, up to negation. Now since $\langle v_s,v_4\rangle=0$, we have that $v_4=e_1-e_2\pm (e_3- e_4)+k'$ for some vector $k'\in\Z^{n-4}$; but then either $\langle v_4,v_2\rangle\neq0$ or $\langle v_4,v_3\rangle\neq0$, both of which are contradictions. Hence $3\not\in V_4$. A similar argument shows that $4\not\in V_3$.
    Hence since $\langle v_4,v_2\rangle=\langle v_4,v_3\rangle=0$, we have $5,6,7,8\in V_4$ and, up to change of basis,  $v_4=e_1-e_2-e_5+e_6+e_7-e_8$. 
    
    Up to reordering, we may now assume that $E_3=\{2,3,5,6\}$ and $a_5\le a_6$. 
    Consider the matrix $M_3$; set $x=a_5-1$ and $y=a_6-1$. Since $\det M_3=0$ and $b=c=5$, it follows from Lemma \ref{lem:55} that
    \begin{equation}
        (x,y)\in\{(1,5), (2,2)\}
        \label{eq:sys}
    \end{equation}
     Since $\langle v_5, v_2\rangle=0$, we have $4\in V_5$, $5\in V_5$, or $6\in V_5$. First suppose $4\not\in V_5$; then $5\in V_5$ or $6\in V_5$, but not both. Up to change of basis, we may assume the former and so (up to negation) $v_5=e_3+e_5+l$ for some vector $l\in\Z^{n-6}$. Since $\langle v_5, v_4\rangle=0$, we must have $7\in V_5$ or $8\in V_5$, but not both. Up to change of basis, we may assume the former so that $v_5=e_3+e_5+e_7+l'$ for some vector $l'\in\Z^{n-7}$.
     It follows from equation (\ref{eq:sys}) that $l'=0$ and $a_6=y+1=3$. Since $\langle v_6,v_5\rangle=0$, we must have $5\in V_6$ or $7\in V_6$, but not both; up to change of basis, we may assume the former so that (up to negation) $v_7=e_3-e_5+e_i$ for some $i$. But then $\langle v_7,v_2\rangle\neq0$, which is a contradiction.    
    Hence we have $4\in V_5$; similarly, $4\in V_6$.
    
    If $\langle v_5,e_i\rangle=\langle v_6,e_i\rangle$ for $i=3,4$, it follows that $|V_5\cap V_6|\ge4$ (since $\langle v_5,v_6\rangle=0$), implying that $a_5,a_6\ge 4$; but this contradicts the solution set (\ref{eq:sys}). Hence we may assume, up to negation, that $v_5=e_3+e_4+m$ and $v_6=e_3-e_4+n$ for some vectors $m,n\in\Z^{n-4}$.
    Since $\langle v_6,v_2\rangle=\langle v_6,v_3\rangle=0$, we must have that $v_6=e_3-e_4+e_5-e_6+e_7-e_8+n'$ for some vector $n'\in\Z^{n-8}$. It now follows from (\ref{eq:sys}) that $m=0$ and $n'=0$.

    We now have $E_3=E_4=\{2,3,5,6\}$, $E_5,E_6\subset\{2,4,6\}$, and $E_7,E_8\supset\{3,4,6\}$. Up to reordering, we may assume $5\in V_7$. Then since $\langle v_2,v_7\rangle=0$, we have $6\in V_7$; hence $E_5=E_6=\{2,4,6,7\}$. Since three of the diagonal entries of the matrix $M_5$ are $5$, Lemma \ref{lem:55} implies that $a_7=2$; up to negation, we have that $v_7=e_5+e_6$. Similarly, up to reordering, we have that $v_8=e_7+e_8$ and $E_7=E_8=\{3,4,6,8\}.$
    We have thus shown (up to change of basis of $\Z^n$, and reordering and negating vectors of $S$) that $B'$ is comprised of the following vectors:
    \begin{align*}
        v_1&=e_1+e_2\\
        v_2&=e_1-e_2-e_3+e_4+e_5-e_6\\
        v_3&=e_1-e_2+e_3-e_4-e_7+e_8\\
        v_4&=e_1-e_2-e_5+e_6+e_7-e_8\\
        v_5&=e_3+e_4\\
        v_6&=e_3-e_4+e_5-e_6+e_7-e_8\\
        v_7&=e_5+e_6\\
        v_8&=e_7+e_8
    \end{align*} 
    Hence $B'$ is the second matrix listed in Theorem \ref{thm:p4=n}.\qed

\subsection{Proof of Theorem \ref{thm:p4notcb}}

Let $S=\{v_1,\ldots,v_n\}\subset\Z^n$ be an orthogonal subset with $p_4(S)=n$ and let $B=\begin{bmatrix} v_1&\cdots&v_n\end{bmatrix}.$ By Lemma \ref{lem:equation}, we have
    $$I(S)=n+ \sum_{\substack{\alpha, i\\ \langle v_\alpha, e_i \rangle \neq 0}}(\langle v_\alpha, e_i \rangle ^2 - 1).$$
    If $|\langle v_u,e_i\rangle|\ge 2$ for some integers $u,i$, then by Corollary \ref{cor:wuobstruction}, $S$ is not cubiquitous, and we are done. Assume instead $|\langle v_u,e_i\rangle|\le 1$ for all integers $u,i$. We thus have that $I(S)=n$ and so $B$ is block diagonal whose blocks are of the form listed in Theorem \ref{thm:p4=n}. 
    By Lemma \ref{lem:block}, we need only show that each of these blocks is not cubiquitous. 
    To this end, let us assume that $B$ is itself one of the blocks.
    
    First suppose that every column of $B$ has exactly four nonzero entries; that is $|V_\alpha|=4$ for all $\alpha$. Then $B^TB=4I$. It follows that $\det(B^TB)=4^n$ and so $\det B=2^n$. By Theorem \ref{thm:greeneowens1}, $S$ is cubiquitious if and only if the sublattice $\Lambda$ generated by $S$ has a Haj{\'o}s basis.
    If $\Lambda$ has a Haj{\'o}s basis, then for every vector $x\in\Lambda$, the first nonzero entry of $x$ is even; this follows from the fact that every vector in $\Lambda$ is a linear combination of the Haj{\'o}s basis vectors, whose first nonzero entries are $2$.  However, since the entries of $v_1$, for example, are either $-1,0,$ or $1$, $\Lambda$ does not have a Haj{\'o}s basis and hence is not cubiquitous.

    Finally, let $B$ be either of the two remaining blocks given by Theorem \ref{thm:p4=n}. Since $\det B<2^n$, we cannot appeal to Theorem \ref{thm:greeneowens1}. However, a quick computer calculation shows that for any vector $y$ in the unit cube $C=(-2,1,1,1,1,1,1,1)+\{0,1\}^8$, the system $Bx=y$ has no solution in $\Z^8$. Hence $C$ does not contain a point of the sublattice generated by $S$ and so $S$ is not cubiquitous.\qed

\section{Cubiquitous Orthogonal Subsets}\label{sec:proof}

Let $S$ be an orthogonal subset. 
Note that if there exist integers $i$ and $s$ such that $V_s=\{i\}$, then by orthogonality $E_i=\{s\}$. We call the subset $$S\setminus\{v_s\}\subset\Z^{n-1}=\langle e_1,\ldots,e_{i-1},e_{i+1},\ldots,e_n\rangle$$
a \textit{projection} of $S$.
Note that after performing all possible projections, we obtain an orthogonal subset with $|V_s|\ge2$ for all $s$.

\begin{lemma} If $S=\{v_1,\ldots,v_n\}\subset\Z^n$ is orthogonal with $|V_\alpha|\ge 2$ for all $\alpha$, then $p_1(S) = 0$.
\label{lem:p1=0}
\end{lemma}

\begin{proof} Assume $p_1(S)\neq 0$. Then there exist integers $i$ and $s$ such that $E_i=\{s\}$; up to reordering, we may assume that $s=n$. Hence $v_n=xe_i+f$ for some $x\in\Z\setminus\{0\}$ and some vector $f\in\Z^{n-1}$. Moreover, since $|V_n|\ge2$, $f\neq 0$. Notice that the subset $$S'=(S\setminus\{v_n\})\cup\{f\}\subset\Z^{n-1}$$
is orthogonal. Label the vectors of $S'$ by: $v'_\alpha=v_\alpha$ for all $\alpha\neq n$ and $v_n'=f$. 
Let $B=\begin{bmatrix} v_1'&\cdots &v_n'\end{bmatrix}$ and $Q=B^TB$. Then the $(s,t)-$th entry of $Q$ is given by $\langle v'_s,v'_t\rangle$; that is, $Q$ is the diagonal matrix with diagonal entries $\langle v_1',v_1'\rangle,\ldots,\langle v_n',v_n'\rangle$. Since $\det Q=\prod_{i=1}^n\langle v_\alpha',v_\alpha'\rangle\neq 0$, $S'$ contains $n$ linearly independent vectors in $\Z^{n-1}$, which is impossible.\end{proof}

\begin{lemma}
Let $S=\{v_1,\ldots,v_n\}\subset\Z^n$ be orthogonal with $|V_\alpha|\ge 2$ for all $\alpha$. Let $i\in\mathcal{P}_2$ such that $E_i = \{s, t\}$ and $|\langle v_s,e_i\rangle|=|\langle v_t,e_i\rangle|=1$. Then up to a change of basis, $v_s = e_i + e_j$ and $v_t = e_i - e_j$ for some $j\in \mathcal{P}_2$.
\label{lem:p2}
\end{lemma}

\begin{proof}
By Lemma \ref{lem:p1=0}, we have that $p_1(S)=0$.
Up to relabeling, we may assume that $s=n-1$ and $t=n$. Up to negation, $v_{n-1}=e_i+f$ and $v_n=e_i+g$ for some nonzero vectors $f,g\in\bZ^n$ satisfying $\langle f,g\rangle=-1$. Consider the subset 
$$S'=(S\setminus\{v_{n-1}, v_n\})\cup\{f,g\}\subset\Z^{n-1}.$$
Label the vectors of $S'$ by: $v'_\alpha=v_\alpha$ for all $\alpha\not\in\{ n-1,n\}$, $v_{n-1}'=f$, and $v_n'=g$.
Then $\langle v'_{n-1}, v'_n\rangle=-1$ and $\langle v'_\alpha,v'_\beta\rangle=0$ for all $\alpha\neq\beta$ where $\beta\neq n$.  
Let $B=\begin{bmatrix} v_1'&\cdots &v_n'\end{bmatrix}$ and $Q=B^TB$. Then the $(s,t)-$th entry of $Q$ is given by $\langle v'_s,v'_t\rangle$; that is, $Q$ is the block diagonal matrix
$$Q=\begin{bmatrix} \langle v_1,v_1\rangle & & & & \\  &\ddots& & &\\ & & \langle v_{n-2},v_{n-2}\rangle & &\\ & & & \langle f,f\rangle & -1\\ &&& -1& \langle g,g\rangle \end{bmatrix}.$$
 Notice that $\det Q=(\langle f,f\rangle\langle g,g\rangle-1)\prod_{\alpha=1}^{n-2}\langle v_\alpha, v_\alpha\rangle$.
If $\langle f, f\rangle\ge2$ or $\langle g,g\rangle \ge 2$, then $\det Q\neq0$; it follows that $S'$ consists of $n$ linearly independent vectors in $\Z^{n-1}$, which is not possible. Thus $\langle f, f\rangle=\langle g,g\rangle=1$ and so, up to a change of basis, $v_s=e_i+e_j$ and $v_t=e_i-e_j$ for some integer $j\neq i$. Finally by orthogonality, $j\in \mathcal{P}_2$.
\end{proof}

For $\alpha\ge2$, let $\mathcal{Q}^S_\alpha$ be the subset of $\mathcal{P}^S_{\alpha}$ defined by $$\mathcal{Q}^S_\alpha=\{i\in\mathcal{P}^S_\alpha\,:\, |\langle v_{u},e_i\rangle|\ge2\text{ for some } u\in E_i^S\}.$$ 
As usual, we will suppress the superscript in the notation if the subset being considered is understood.

For $i\in\mathcal{P}_2$, set $E_i:=\{s(i),t(i)\}$. 
Suppose $i\in \mathcal{P}_2\setminus\mathcal{Q}_2$. Then by Lemma \ref{lem:p2}, $v_{s(i)}=e_{i}+e_j$ and $v_{t(i)}=e_i-e_j$ for some $j\in\mathcal{P}_2$. We call the subset
$$S\setminus\{v_{s(i)}, v_{t(i)}\}\subset\Z^{n-2}=\langle e_1,\ldots,e_{i-1},e_{i+1},\ldots,e_{j-1},e_{j+1},\ldots,e_n\rangle$$
a \textit{double projection} of $S$. 
Note that after performing all possible double projections, we obtain an orthogonal subset with $\mathcal{P}_2=\mathcal{Q}_2$.

\begin{proposition}
 Let $S=\{v_1,\ldots,v_n\}\subset\Z^n$ be an orthogonal subset with $|V_\alpha|\ge2$ for all $\alpha$ and $\mathcal{P}_2=\mathcal{Q}_2$. Then $S$ is not cubiqutious.
 \label{prop:p2}
\end{proposition}

\begin{proof}
First note that by Lemma \ref{lem:p1=0}, since $|V_\alpha|\ge2$ for all $\alpha$, $p_1(S)=0$.
Since $\sum_{j=1}^np_j(S)=n\ge1$, we have that $p_j(S)\ge1$ for some $j\ge2$. 

Since $\mathcal{P}_2=\mathcal{Q}_2$, up to relabeling, we may assume that $|\langle v_{s(i)},e_i\rangle|\ge2$ for all $i\in\mathcal{P}_2$. Hence
\begin{equation}
\sum_{i\in\mathcal{P}_2}(\langle v_{s(i)},e_i\rangle^2-1)\ge 3p_2(S).
\label{eq:inequality1}
\end{equation}

\noindent Next note that
\begin{equation}
\begin{split}
|R_o|&\ge |\{i\in\mathcal{P}_3\,:\,\text{$k_i$ is odd}\}|\\
&\ge |\{i\in\mathcal{P}_3\,:\,\text{$|\langle v_u,e_i\rangle|= 1$ for all $u\in E_i$}\}|\\
&= |\mathcal{P}_3\setminus\mathcal{Q}_3|\\
&\ge p_3(S)-\sum_{\substack{i\in \mathcal{P}_3\\u\in E_i}}(\langle v_u,e_i\rangle^2-1).
\end{split}
\label{eq:inequality2}
\end{equation}

\noindent Now by Lemmas \ref{lem:equation} and \ref{lem:p1=0}, and the fact that $\sum_{i=2}^{n}p_j(S)=n$, we have that 
\begin{equation}
\begin{split}
       I(S) &=-p_2(S)+ \sum_{j=4}^{n}(j-3)p_j(S) + \sum_{\substack{\alpha, i\\ \langle v_\alpha, e_i \rangle \neq 0}}(\langle v_\alpha, e_i \rangle ^2 - 1)\\
        &= n-2p_2(S)-p_3(S)+\sum_{j=4}^{n}(j-4)p_j(S) + \sum_{\substack{\alpha, i\\ \langle v_\alpha, e_i \rangle \neq 0}}(\langle v_\alpha, e_i \rangle ^2 - 1)\\
       &\ge n-\Big(2p_2(S)- \sum_{i\in\mathcal{P}_2}(\langle v_{s(i)},e_i\rangle^2-1)\Big)-\Big(p_3(S)-\sum_{\substack{i\in\mathcal{P}_3\\ u\in E_i}}(\langle v_{u},e_i\rangle^2-1)\Big)\\
       &\qquad\qquad\qquad\qquad\qquad\qquad\qquad\qquad\quad+\sum_{\substack{i\in\mathcal{P}_4\\ u\in E_i}}(\langle v_{u},e_i\rangle^2-1)+\sum_{j=4}^{n}(j-4)p_j(S)\\
       &\ge n+p_2(S)-|R_o|+\sum_{\substack{i\in\mathcal{P}_4\\ u\in E_i}}(\langle v_{u},e_i\rangle^2-1)+\sum_{j=4}^{n}(j-4)p_j(S),\\
    \end{split}
    \label{eq:inequalitystring}
\end{equation}
where the final inequality follows from inequalities (\ref{eq:inequality1}) and (\ref{eq:inequality2}).
If $p_2(S)=|\mathcal{Q}_2|\neq 0$ or $p_j(S)\neq 0$ for some $j> 4$, then the last expression is strictly greater than $n-|R_o|$, implying that $I(S)>n-3|R_o|$; hence by Corollary \ref{cor:wuobstruction}, $S$ is not cubiquitous, and we are done. 
If $|R_o|>0$, then $I(S)\ge n-|R_o|>n-3|R_o|$, implying once again that $S$ is not cubiquitious.  

We now suppose that $|\mathcal{Q}_2|=0$, $p_j(S)=0$ for all $j>4$, and $|R_o|=0$; hence $p_3(S)+p_4(S)=n$. If $p_4(S)=n$, then $S$ is not cubiquitous by Theorem \ref{thm:p4notcb}. Now suppose $p_4(S)\neq n$ so that that $p_3(S)\neq 0$. Since $|R_o|=0$, we have that $\mathcal{P}_3=\mathcal{Q}_3$; that is, for all $i\in \mathcal{P}_3$, there exists $u$ such that $|\langle v_u,e_i\rangle|\ge 2$. It follows that 
\begin{equation}
\sum_{i\in\mathcal{P}_3}(\langle v_{s(i)},e_i\rangle^2-1)\ge 3p_3(S).
\label{eq:inequality4}
\end{equation}
Starting with the third line of inequality (\ref{eq:inequalitystring}) and applying inequality (\ref{eq:inequality4}), we have that 
\begin{align*}
    I(S)&\ge n-\Big(p_3(S)-\sum_{\substack{i\in\mathcal{P}_3\\ u\in E_i}}(\langle v_{u},e_i\rangle^2-1)\Big)+\sum_{\substack{i\in\mathcal{P}_4\\ u\in E_i}}(\langle v_{u},e_i\rangle^2-1)\\
    &\ge n-\Big(p_3(S)-\sum_{\substack{i\in\mathcal{P}_3\\ u\in E_i}}(\langle v_{u},e_i\rangle^2-1)\Big)\\
    &\ge n+2p_3(S)\\
    &>n\\
    &=n-3|R_o|.
\end{align*}
It follows from Corollary \ref{cor:wuobstruction} that $S$ is cubiquitous. 
\end{proof}

We are now ready to prove Theorem \ref{thm:main}.

\begin{proof}[Proof of Theorem \ref{thm:main}] 
First note that if $\Lambda$ is an orthogonal sublattice with basis as described in the statement of Theorem \ref{thm:main}, then it is the direct sum of cubiquitous sublattices; it follows from Lemma \ref{lem:block} that $\Lambda$ is cubiquitous.
Now suppose $S=\{v_1,\ldots,v_n\}$ is a cubiquitous orthogonal subset and let $\Lambda$ be the lattice generated by $S$. Let $B=\begin{bmatrix}v_1&\cdots& v_n\end{bmatrix}$. 
Perform all possible projections and double projections of $S$ to obtain an orthogonal subset $S'$ with $p_1(S')=0$ (by Lemma \ref{lem:p1=0}) and $\mathcal{P}^{S'}_2=\mathcal{Q}^{S'}_2$. If $S'$ is nonempty, then by Lemma \ref{lem:block}, $S'$ is also cubiquitous. However, this directly contradicts Proposition \ref{prop:p2}, which tells us that $S'$ is not cubiquitous. Thus $S'$ is necessarily empty.
It follows that $p_j(S)=0$ for all $j\ge 3$ and $|\mathcal{Q}^S_2|=0$. By Lemma \ref{lem:p2}, for all $i\in\mathcal{P}^S_2$ we have that $v_{s(i)}=e_i-e_j$ and $v_{t(i)}=e_i+e_j$ for some $j\in\mathcal{P}^S_2$ (up to a change of basis). Moreover, if $i\in \mathcal{P}^S_1$ and $E_i=\{s\}$, then $v_s=ae_i$ for some integer $a$. Thus, up to swapping columns, 
 $B$ is a block matrix with blocks of the form $[a]$ and $\begin{bmatrix}1&-1\\1&1\end{bmatrix}$; that is, $\Lambda$ is a direct sum of sublattices with bases of the form $\{[a]\}$ and $\{[1,1]^T, [-1,1]^T\}$.
Since $\text{Span}\{[a]\}$ is not cubiquitous when $|a|\ge 3$, it follows that each $1\times 1$ block must be of the form $[\pm1]$ or $[\pm2]$. Up to a change of basis, the result follows.\end{proof}

\section{Application to alternating connected sums of torus links}\label{sec:toruslinks}

In this section we consider the applications of our lattice embedding results to the connected sums of torus links, and discuss some related questions. 
A connected sum of links is alternating if and only if each link is alternating (\cite[Theorem 1]{Menasco}). Moreover, the only alternating torus links are $T_{2,k}$ for some integer $k$ (see for example, \cite[Theorem 7.3.2]{murasugi}), whose double branched cover is $\Sigma(T_{2,k}) = L(k,1).$ 
We will assume throughout that each torus knot is negative; that is it is of the form $T(2,k)$, where $k<1$.

In general, the connected sum of two links $L_1$ and $L_2$ depends on the components of each link that are summed together. However, the double branched cover of both of these connected sums are diffeomorphic to $\Sigma(L_1)\#\Sigma(L_2)$. 
Let $L$ be a connected sum of the torus links $T(2,k_1), \ldots, T(2,k_n)$. Note that if we connect sum another alternating torus link $T(2,k_{n+1})$ to $L$, then the choice of component of $T(2,k_{n+1})$ does not matter; each choice will yield isotopic links (if we do not order the link components of $T(2,k_{n+1})$).
Let $K$ and $K'$ be components of $L$ and consider the connected sums $L_K$ and $L_{K'}$ obtained by connect summing $T(2,k_{n+1})$ to $K$ and $K'$, respectively. We will see that lattices induced by $L_K$ and $L_{K'}$ are isometric; hence we may focus on the lattices induced by the connected sums of the form given in Figure \ref{fig:orthoa}.

 \begin{figure}
\centering
\begin{overpic}
     [scale=.7]{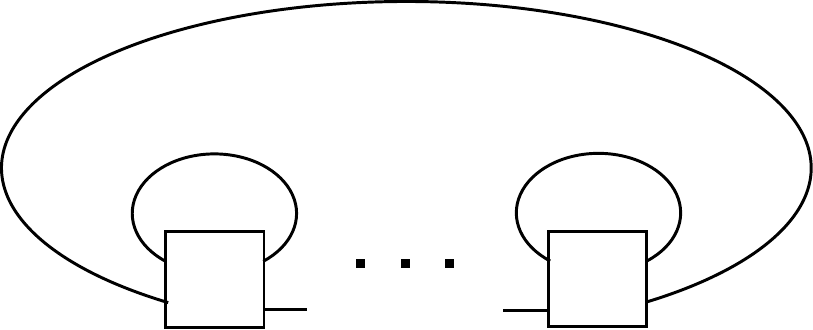}
     \put(24,4){\Large{$k_1$}}
     \put(71,4){\Large{$k_n$}}
\end{overpic}
     \caption{A connected sum of alternating torus links.  Each integer $k_i$ indicates the number of right-handed half twists. }
     \label{fig:orthoa}
\end{figure}

\begin{proposition}[c.f. \cite{greenemutation}] Let $B_{K}$ and $B_{K'}$ denote the black checkerboard surfaces for $L_K$ and $L_{K'}$, respectively. Then $\Sigma(B_K)$ is diffeomorphic to $\Sigma(B_{K'})$.
\label{prop:diffeo}
\end{proposition}

\begin{proof}
By convention, $B_K$ and $B_{K'}$ are bounded. Let $\Gamma_B$ and $\Gamma_{B'}$ denote the respective Tait graphs of $B$ and $B'$. 
Let $v_a$ and $v_b$ denote the vertices of the Tait graph for $L$ corresponding to $K$ and $K'$. Let $-a$ and $-b$ denote the weights of these vertices. Consider the Tait graphs of $T(2,k_{n+1})$ and $L$ shown in Figure \ref{fig:taita}. The Tait graphs $\Gamma_B$ and $\Gamma_{B'}$ are then the graphs shown in Figure \ref{fig:taitb}. To produce handlebody diagrams for $\Sigma(B)$ and $\Sigma(B')$, we must first delete a vertex from each graph. Deleting the vertices corresponding to $v_a$ in both graphs, yields the graphs in Figure \ref{fig:taitc}. It is now easy to see that the obvious handlebody diagram for $\Sigma(B')$ can be obtained from the obvious handlebody diagram of $\Sigma(B)$ by sliding the 2-handle corresponding to $v_b$ over the 2-handle corresponding to the isolated vertex with weight $-k_{n+1}$. The result follows.
\end{proof}

\begin{figure}
\centering
\begin{subfigure}{.9\textwidth}
\centering
\begin{tikzpicture}[dot/.style = {circle, fill, minimum size=1pt, inner sep=0pt, outer sep=0pt}]
\tikzstyle{smallnode}=[circle, inner sep=0mm, outer sep=0mm, minimum size=2mm, draw=black, fill=black];

\node[smallnode, label={[label distance=.1cm]90:$-k_{n+1}$}] (1) at (0,0) {};
\node[smallnode, label={[label distance=.1cm]90:$-k_{n+1}$}] (2) at (2,0) {};

\node[smallnode, label={[label distance=0cm]90:$-a$}] (a) at (4,0) {};
\node[smallnode, label={[label distance=0cm]90:$-b$}] (b) at (8,0) {};
\node[rectangle,draw,minimum width = 2cm, minimum height = 2cm] (r) at (6,0) {T};

\node (c) at (4.7, .4){};
\node (d) at (4.7, -.4){};
\node (e) at (7.3, .4){};
\node (f) at (7.3, -.4){};
\node (g) at (1, .3){};
\node (h) at (1, -.3){};

\draw[-] (1) to [out=30, in=150] (2);
\draw[-] (1) to [out=-30, in=210] (2);
\draw[-] (a) to [out=30, in=160](r);
\draw[-] (a) to [out=-30, in=200](r);
\draw[-] (b) to [out=210, in=-20] (r);
\draw[-] (b) to [out=150, in=20] (r);
\draw[dotted, thick] (c) -- (d);
\draw[dotted, thick] (e) -- (f);
\draw[dotted, thick] (g) -- (h);
\end{tikzpicture}
\caption{Tait graphs of $T(2,k_{n+1})$ and $L$}
\label{fig:taita}
\end{subfigure}

\begin{subfigure}{0.9\textwidth}
\centering
\vspace{.5cm}
\begin{tikzpicture}[dot/.style = {circle, fill, minimum size=1pt, inner sep=0pt, outer sep=0pt}]
\tikzstyle{smallnode}=[circle, inner sep=0mm, outer sep=0mm, minimum size=2mm, draw=black, fill=black];

\node[smallnode, label={[label distance=0cm]90:$-k_{n+1}$}] (1) at (2,0) {};
\node[smallnode, label={[label distance=.1cm]270:$-a-k_{n+1}$}] (a) at (4,0) {};
\node[smallnode, label={[label distance=-.1cm]90:$-b$}] (b) at (8,0) {};
\node[rectangle,draw,minimum width = 2cm, minimum height = 2cm] (r) at (6,0) {T};
\node (c) at (4.7, .4){};
\node (d) at (4.7, -.4){};
\node (e) at (7.3, .4){};
\node (f) at (7.3, -.4){};
\node (g) at (3, .3){};
\node (h) at (3, -.3){};

\draw[-] (1) to [out=30, in=150] (a);
\draw[-] (1) to [out=-30, in=210] (a);
\draw[-] (a) to [out=30, in=160](r);
\draw[-] (a) to [out=-30, in=200](r);
\draw[-] (b) to [out=210, in=-20] (r);
\draw[-] (b) to [out=150, in=20] (r);
\draw[dotted, thick] (c) -- (d);
\draw[dotted, thick] (e) -- (f);
\draw[dotted, thick] (g) -- (h);
\node[smallnode, label={[label distance=0cm]90:$-k_{n+1}$}] (11) at (15,0) {};
\node[smallnode, label={[label distance=0cm]90:$-a$}] (a1) at (9,0) {};
\node[smallnode, label={[label distance=.1cm]270:$-b-k_{n+1}$}] (b1) at (13,0) {};
\node[rectangle,draw,minimum width = 2cm, minimum height = 2cm] (r1) at (11,0) {T};
\node (c1) at (9.7, .4){};
\node (d1) at (9.7, -.4){};
\node (e1) at (12.3, .4){};
\node (f1) at (12.3, -.4){};
\node (g1) at (14, .3){};
\node (h1) at (14, -.3){};

\draw[-] (b1) to [out=30, in=150] (11);
\draw[-] (b1) to [out=-30, in=210] (11);
\draw[-] (a1) to [out=30, in=160](r1);
\draw[-] (a1) to [out=-30, in=200](r1);
\draw[-] (b1) to [out=210, in=-20] (r1);
\draw[-] (b1) to [out=150, in=20] (r1);
\draw[dotted, thick] (c1) -- (d1);
\draw[dotted, thick] (e1) -- (f1);
\draw[dotted, thick] (g1) -- (h1);
\end{tikzpicture}
\caption{The Tait graphs $\Gamma_B$ and $\Gamma_{B'}$}\label{fig:taitb}
\end{subfigure}

\begin{subfigure}{0.9\textwidth}
\centering
\vspace{.5cm}
\begin{tikzpicture}[dot/.style = {circle, fill, minimum size=1pt, inner sep=0pt, outer sep=0pt}]
\tikzstyle{smallnode}=[circle, inner sep=0mm, outer sep=0mm, minimum size=2mm, draw=black, fill=black];

\node[smallnode, label={[label distance=-.1cm]90:$-k_{n+1}$}] (1) at (3,0) {};
\node[smallnode, label={[label distance=0cm]90:$-b$}] (b) at (8,0) {};
\node[rectangle,draw,minimum width = 2cm, minimum height = 2cm] (r) at (6,0) {T};
\node (e) at (7.3, .4){};
\node (f) at (7.3, -.4){};

\draw[-] (b) to [out=210, in=-20] (r);
\draw[-] (b) to [out=150, in=20] (r);
\draw[dotted, thick] (e) -- (f);

\node[smallnode, label={[label distance=0cm]90:$-k_{n+1}$}] (11) at (15,0) {};
\node[smallnode, label={[label distance=.1cm]270:$-b-k_{n+1}$}] (b1) at (13,0) {};
\node[rectangle,draw,minimum width = 2cm, minimum height = 2cm] (r1) at (11,0) {T};
\node (e1) at (12.3, .4){};
\node (f1) at (12.3, -.4){};
\node (g1) at (14, .3){};
\node (h1) at (14, -.3){};

\draw[-] (b1) to [out=30, in=150] (11);
\draw[-] (b1) to [out=-30, in=210] (11);
\draw[-] (b1) to [out=210, in=-20] (r1);
\draw[-] (b1) to [out=150, in=20] (r1);
\draw[dotted, thick] (e1) -- (f1);
\draw[dotted, thick] (g1) -- (h1);

\end{tikzpicture}
\caption{Graphs giving handlebody diagrams of $\Sigma(B)$ and $\Sigma(B')$}\label{fig:taitc}
\end{subfigure}

\label{fig:linking}
\caption{}
\end{figure}

\begin{corollary}
Let $L$ be a connected sum of negative torus links $T(2,k_1),\ldots,T(2,k_n)$. Then $\Sigma(L)$ bounds a rational ball if and only if $k_i\in\{-1,-2, -4\}$ for all $i$ and there is an even number of parameters equal to $-2$.
\label{cor:torussum}
\end{corollary}
\begin{proof}
    By Proposition \ref{prop:diffeo}, the lattice embedding associated to $L$ does not depend on the way one performs the connected sum. Consider the specific connected sum diagram depicted in Figure \ref{fig:orthoa}. Denote by $W$ the unbounded checkerboard surface of $L$. Note that if we delete the vertex associated to the unbounded region in the orresponding Tait graph $\Gamma_W$, the resulting diagram consists of isolated vertices with weight $k_i$ for $i=1,\cdots, n.$ It follows that $\Lambda(L)$ has matrix representation $\operatorname{diag}\{k_1,\cdots,k_n\}$.

    If $\Sigma(L)$ bounds a rational homology 4-ball, then by Theorem \ref{thm:main} and Theorem \ref{thm:greeneowens} we have that $k_i\in\{-1,-2,-4\}$ for all $i$ and there is an even number of parameters equal to $-2$. On the other hand, to show that if the parameters are as stated, then $\Sigma(L)$ bounds a rational homology 4-ball, one can show that the the link shown in Figure \ref{fig:orthoa} bounds a smooth, properly embedded surface in $B^4$ made up of a disk and a Mobius band; the double branched cover of such a link necessarily bounds a rational homology 4-ball. See \cite{lisca-sumsoflensspaces} for details. 
\end{proof}

\begin{proof}[Proof of Theorem \ref{thm:maintop}]
 The $``\Longrightarrow"$ direction follows from Theorem \ref{thm:greeneowens}. For the $``\Longleftarrow"$ direction,  let $L$ be an alternating connected sum of $n$ negative torus links. By the discussion above, these torus links are of the form $T(2,k_1),\ldots,T(2,k_n)$ for integers $k_i \geq 1$. 
 The result now follows from Corollary \ref{cor:torussum}. 
 \end{proof}

\section{A note on cubiquity and contractions}\label{sec:contractions}

The notion of \textit{contractions} has been used in the literature (\cite{liscalensspaces}, \cite{lisca-sumsoflensspaces}, \cite{simone}) to help understand the behavior of particular subsets, which often belong to infinite families. We will see in this section, that contractions preserve the Wu obstruction (Proposition \ref{prop:Wuobstruction}). This, in turn, allows us to obstruct cubiquity for a large class of subsets.

\begin{definition} Suppose $S=\{v_1,\ldots,v_n\}\subset\Z^n$, $n\ge 3$, such that there exist integers $i,s$, $t$, and $u$ such that $E_i=\{s,t,u\}$, where  $\langle v_{s}, v_{t}\rangle=-1$, $\langle v_s, e_i\rangle=-\langle v_t, e_i\rangle=\pm1$, $|\langle v_u, e_i\rangle|=1$, and $a_u\ge 3$. Then the subset $S'\subset \mathbb{Z}^{n-1}=\langle e_1,\ldots,e_{i-1},e_{i+1},\ldots, e_n\rangle$ defined by
\begin{center}	
$S'=(S\setminus\{v_s,v_{t},v_u\})\cup\{v_s+v_{t},v_u-\langle v_u, e_i\rangle e_i\},$
\end{center}
is called a \textit{contraction} of $S$. 
\label{def:contraction}
\end{definition}

\contractions*

\begin{proof} Let $S$ be a subset with Wu element $W=\sum_{\alpha=1}^nk_\alpha e_\alpha$.
Let $i,s,t$, $u$, and $S'$ be as in Definition \ref{def:contraction}. Let $W'=\sum_{\alpha=1, \alpha\neq i}^nk_\alpha'e_\alpha$ denote the Wu element of $S'$. It is easy to see that $k_i'=k_i$ for all $i\neq \alpha$ and hence $W=W'\pm e_i$. It follows that $\sum_{\alpha=1}^nk_\alpha>4n-3|R_o^S|$ if and only if $\sum_{\alpha=1, \alpha\neq i}^nk'_\alpha>4n-3|R_o^{S'}|.$ In particular, the Wu obstruction given by Proposition \ref{prop:Wuobstruction} is preserved by contractions.
\end{proof}

We will now restrict our attention to good subsets as defined by Lisca in \cite{liscalensspaces}. It is worth mentioning that Lisca defines contractions differently than the definition presented here; however these definitions coincide for the subsets considered.

\begin{proposition} If $S$ is a good subset with $p_1(S) \neq 0$ then $S$ is cubiquitous if and only if $I(S)\le 0$.  
\label{prop:goodcub}
\end{proposition}

\begin{proof}
By Lemma 3.2 in \cite{liscalensspaces}, any good subset with $p_1(S)>0$ can be contracted to a length 3 subset with $p_1(S)>0$. It can be checked by hand that the only such subsets are of the form 
$$S'=\{e_1+e_2, -e_2+xe_3, e_2-e_1\},$$
where $x\neq 0$. The Wu element of this subset is $W=xe_3+e_2$. If $I(S)\le 0$, then $|x|\ge3$; by Proposition \ref{prop:Wuobstruction}, $S'$ is not cubiquitous. Hence $S$ is not cubiquitous. If $I(S)>0$, then we have either $|x|=1$ or $|x|=2$. It is known (see \cite{liscalensspaces}) that in both cases, $S$ corresponds to a lens space that bounds a rational homology 4-ball. It follows that $S$ is cubiquitous.
\end{proof}

\begin{remark}
The proof strategy of Proposition \ref{prop:goodcub} can be applied to good subsets with $p_1(S)=0$, but a general argument is not as clear. In the $p_1(S)\neq 0$ case, good subsets can be contracted to length 3 subsets; this is not the case when $p_1(S)= 0$. In fact, it is unknown in general how small the length of a subset with $p_1=0$ can be made via contractions.
\end{remark}

\printbibliography
\end{document}